\documentclass{amsart}
\usepackage{amsfonts}

\newtheorem{theorem}{Theorem}
\theoremstyle{plain}

\newtheorem{definition}{Definition}

\newtheorem{lemma}{Lemma}

\newtheorem{remark}{Remark}

\numberwithin{equation}{section}
\input{tcilatex}

\begin{document}
\title[Ostrowski type inequalities for s-convex functions via fractional
integrals]{New inequalities of Ostrowski type for mappings whose derivatives
are $s-$convex in the second sense via fractional integrals}
\author{Erhan SET}
\address{Department of Mathematics, \ Faculty of Science and Arts, D\"{u}zce
University, D\"{u}zce-TURKEY}
\email{erhanset@yahoo.com}

\begin{abstract}
New identity similar to an identity of \cite{ADDC} for fractional integrals
have been defined. Then making use of this identity, some new Ostrowski type
inequalities for Riemann-Liouville fractional integral have been developed.
Our results have some relationships with the results of Alomari et. al.,
proved in \cite{ADDC} [published in. Appl. Math. Lett. 23 (2010) 1071-1076]
and the analysis used in the proofs is simple.
\end{abstract}

\maketitle

\section{Introduction and Preliminary Results}

In 1938, A.M. Ostrowski proved the following interesting and useful integral
inequality (\cite{OST}, see also \cite[page 468]{Mitrinovic1}):

\begin{theorem}
\label{t1} Let $f:I\rightarrow 
\mathbb{R}
,$ where $I\subseteq 
\mathbb{R}
$ is an interval, be a mapping differentiable in the interior $I^{\circ }$
of $I$, and let $a,b\in I^{\circ }$ with $a<b$. If $\left\vert f^{\prime
}\left( x\right) \right\vert \leq M$ for all $x\in \left[ a,b\right] $, then
the following inequality holds:
\end{theorem}

\begin{equation}
\left\vert f(x)-\frac{1}{b-a}\int_{a}^{b}f(t)dt\right\vert \leq M\left(
b-a\right) \left[ \frac{1}{4}+\frac{\left( x-\frac{a+b}{2}\right) ^{2}}{%
\left( b-a\right) ^{2}}\right]   \label{e1}
\end{equation}%
for all $x\in \left[ a,b\right] .$ The constant $\frac{1}{4}$ is the best
possible in the sense that it cannot be replaced by a smaller one.

This inequality gives an upper bound for the approximation of the integral
average $\frac{1}{(b-a)}\int_{a}^{b}f(t)dt$ by the value $f(x)$ at point $%
x\in \left[ a,b\right] .$ In recent years, such inequalities were studied
extensively by many researchers and numerious generalizations, extensions
and variants of them appeared in a number of papers see (\cite{D1}-\cite%
{ADDC})

In \cite{hudzik}, the class of functions which are $s-$convex in the second
sense has been introduced by Hudzik and Maligranda as the following:

\begin{definition}
\label{d1} A function $f:[0,\infty )\mathbb{\rightarrow R}$ is said to be $%
s- $convex in the second sense if 
\begin{equation*}
f(\lambda x+(1-\lambda )y)\leq \lambda ^{s}f(x)+(1-\lambda )^{s}f(y)
\end{equation*}%
for all $x,y\in \lbrack 0,\infty )$, $\lambda \in \lbrack 0,1]$ and for some
fixed $s\in (0,1].$ This class of $s$-convex functions is usually denoted by 
$K_{s}^{2}.$
\end{definition}

It can be easily seen that for $s=1,$ $s-$convexity reduces to ordinary
convexity of functions defined on $[0,\infty ).$

In \cite{dragomir3}, S.S. Dragomir and S. Fitzpatrick proved a variant of
Hadamard's inequality which holds for $s-$convex functions in the second
sense:

\begin{theorem}
\label{t2} Suppose that $f:[0,\infty )\rightarrow \lbrack 0,\infty )$ is an $%
s$-convex function in the second sense, where $s\in (0,1),$ and let $a,b\in
\lbrack 0,\infty ),$ $a<b.$ If $f^{\prime }\in L^{1}(\left[ a,b\right] ),$
then the following inequalities hold:%
\begin{equation}
2^{s-1}f(\frac{a+b}{2})\leq \frac{1}{b-a}\int\limits_{a}^{b}f(x)dx\leq \frac{%
f(a)+f(b)}{s+1}.  \label{e.1.3}
\end{equation}
\end{theorem}

The constant $k=\frac{1}{s+1}$ is the best possible in the second inequality
in (\ref{e.1.3}).

The following identity is proved by Alomari et.al. (see \cite{ADDC})

\begin{lemma}
\label{l1} Let $f:I\subset \mathbb{R\rightarrow R}$ be a differentiable
mapping on $I^{\circ }$ where $a,b\in I$ with $a<b$. If $f^{\prime }\in L%
\left[ a,b\right] $, then we have the equality: 
\begin{eqnarray*}
f(x)-\frac{1}{b-a}\int_{a}^{b}f(t)dt &=&\frac{\left( x-a\right) ^{2}}{b-a}%
\int_{0}^{1}tf^{\prime }(tx+(1-t)a)dt \\
&&-\frac{\left( b-x\right) ^{2}}{b-a}\int_{0}^{1}tf^{\prime }(tx+(1-t)b)dt
\end{eqnarray*}%
for each $x\in \left[ a,b\right] .$
\end{lemma}

Using the Lemma \ref{l1}, Alomari et al. in \cite{ADDC} established the
following results which holds for s-convex functions in the second sense.

\begin{theorem}
\label{t3} Let $f:I\subset \lbrack 0,\infty )\mathbb{\rightarrow R}$ be a
differentiable mapping on $I^{\circ }$ such that $f^{\prime }\in L\left[ a,b%
\right] ,$ where $a,b\in I$ with $a<b.$ If $\left\vert f^{\prime
}\right\vert $ is s-convex in the second sense on $[a,b]$ for some fixed $%
s\in (0,1]$ and $\ \left\vert f^{\prime }(x)\right\vert \leq M,$ $x\in \left[
a,b\right] $, then we have the inequality: 
\begin{equation}
\left\vert f(x)-\frac{1}{b-a}\int_{a}^{b}f(t)dt\right\vert \leq \frac{M}{b-a}%
\left[ \frac{\left( x-a\right) ^{2}+\left( b-x\right) ^{2}}{s+1}\right] ,
\label{e.1.4}
\end{equation}%
for each $x\in \left[ a,b\right] .$
\end{theorem}

\begin{theorem}
\label{t4} Let $f:I\subset \lbrack 0,\infty )\mathbb{\rightarrow R}$ be a
differentiable mapping on $I^{\circ }$ such that $f^{\prime }\in L\left[ a,b%
\right] ,$ where $a,b\in I$ with $a<b.$ If $\left\vert f^{\prime
}\right\vert ^{q}$ is s-convex in the second sense on $[a,b]$ for some fixed 
$s\in (0,1],$ $q>1,$ $p=\frac{q}{q-1}$ and $\ \left\vert f^{\prime
}(x)\right\vert \leq M,$ $x\in \left[ a,b\right] $, then we have the
inequality: 
\begin{equation}
\left\vert f(x)-\frac{1}{b-a}\int_{a}^{b}f(t)dt\right\vert \leq \frac{M}{b-a}%
\left[ \frac{\left( x-a\right) ^{2}+\left( b-x\right) ^{2}}{s+1}\right] ,
\label{1.4.5}
\end{equation}%
for each $x\in \left[ a,b\right] .$
\end{theorem}

\begin{theorem}
\label{t5} Let $f:I\subset \lbrack 0,\infty )\mathbb{\rightarrow R}$ be a
differentiable mapping on $I^{\circ }$ such that $f^{\prime }\in L\left[ a,b%
\right] ,$ where $a,b\in I$ with $a<b.$ If $\left\vert f^{\prime
}\right\vert ^{q}$ is s-convex in the second sense on $[a,b]$ for some fixed 
$s\in (0,1],$ $q\geq 1$ and $\ \left\vert f^{\prime }(x)\right\vert \leq M,$ 
$x\in \left[ a,b\right] $, then we have the inequality: 
\begin{equation}
\left\vert f(x)-\frac{1}{b-a}\int_{a}^{b}f(t)dt\right\vert \leq M\left( 
\frac{2}{s+1}\right) ^{\frac{1}{q}}\left[ \frac{\left( x-a\right)
^{2}+\left( b-x\right) ^{2}}{2\left( b-a\right) }\right] ,  \label{1.4.6}
\end{equation}%
for each $x\in \left[ a,b\right] .$
\end{theorem}

\begin{theorem}
\label{t6} Let $f:I\subset \lbrack 0,\infty )\mathbb{\rightarrow R}$ be a
differentiable mapping on $I^{\circ }$ such that $f^{\prime }\in L\left[ a,b%
\right] ,$ where $a,b\in I$ with $a<b.$ If $\left\vert f^{\prime
}\right\vert ^{q}$ is s-concave in the second sense on $[a,b]$ for some
fixed $s\in (0,1],$ $q>1$ and $p=\frac{q}{q-1}$ , then we have the
inequality: 
\begin{eqnarray}
&&\left\vert f(x)-\frac{1}{b-a}\int_{a}^{b}f(t)dt\right\vert  \label{1.4.7}
\\
&\leq &\frac{2^{(s-1)/q}}{\left( 1+p\right) ^{1/p}\left( b-a\right) }\left[
\left( x-a\right) ^{2}\left\vert f^{\prime }\left( \frac{x+a}{2}\right)
\right\vert +\left( b-x\right) ^{2}\left\vert f^{\prime }\left( \frac{b+x}{2}%
\right) \right\vert \right] ,  \notag
\end{eqnarray}%
for each $x\in \left[ a,b\right] .$
\end{theorem}

For other recent results concerning s-convex functions see \cite{ADDC}-\cite%
{SSO}.

We give some necessary definitions and mathematical preliminaries of
fractional calculus theory which are used throughout this paper.

\begin{definition}
Let $f\in L_{1}[a,b].$ The Riemann-Liouville integrals $J_{a+}^{\alpha }f$
and $J_{b-}^{\alpha }f$ of order $\alpha >0$ with $a\geq 0$ are defined by 
\begin{equation*}
J_{a+}^{\alpha }f(x)=\frac{1}{\Gamma (\alpha )}\int_{a}^{x}\left( x-t\right)
^{\alpha -1}f(t)dt,\ \ x>a
\end{equation*}%
and%
\begin{equation*}
J_{b-}^{\alpha }f(x)=\frac{1}{\Gamma (\alpha )}\int_{x}^{b}\left( t-x\right)
^{\alpha -1}f(t)dt,\ \ x<b,
\end{equation*}
where $\Gamma (\alpha )=\int_{0}^{\infty }e^{-t}u^{\alpha -1}du$. Here, $%
J_{a+}^{0}f(x)=J_{b-}^{0}f(x)=f(x).$
\end{definition}

In the case of $\alpha =1,$ the fractional integral reduces to the classical
integral. Some recent results and properties concerning this operator can be
found (\cite{Anastassiou}-\cite{sarikaya2}).

In spired and motivated by the recent results \ given in \cite{ADDC}, \cite%
{Anastassiou}-\cite{Dahmani4} and \cite{sarikaya}, in the present note, we
establish new Ostrowski type inequalities for $s-$convex functions in the
second sense via Riemann-Liouville fractional integral. An interesting
feature of our results is that they provide new estimates on these types of
inequalities for fractional integrals.

\section{Ostrowski Type Inequalities via Fractional Integrals}

In order to prove our main results we need the following identity:

\begin{lemma}
\label{L1} Let $f:\left[ a,b\right] \rightarrow \mathbb{R}$ be a
differentiable mapping on $(a,b)$ with $a<b.$ If $f^{\prime }\in L\left[ a,b%
\right] ,$ then for all $x\in \left[ a,b\right] $ and $\alpha >0$ we have:%
\begin{eqnarray}
&&\left( \frac{\left( x-a\right) ^{\alpha }+\left( b-x\right) ^{\alpha }}{b-a%
}\right) f(x)-\frac{\Gamma (\alpha +1)}{\left( b-a\right) }\left[
J_{x-}^{\alpha }f(a)+J_{x+}^{\alpha }f(b)\right]  \notag \\
&&  \label{E1} \\
&=&\frac{\left( x-a\right) ^{\alpha +1}}{b-a}\int_{0}^{1}t^{\alpha
}f^{\prime }\left( tx+(1-t)a\right) dt-\frac{\left( b-x\right) ^{\alpha +1}}{%
b-a}\int_{0}^{1}t^{\alpha }f^{\prime }\left( tx+(1-t)b\right) dt  \notag
\end{eqnarray}%
where $\Gamma (\alpha )=\int_{0}^{\infty }e^{-t}u^{\alpha -1}du$.
\end{lemma}

\begin{proof}
\bigskip By integration by parts, we can state 
\begin{eqnarray}
&&\int_{0}^{1}t^{\alpha }f^{\prime }\left( tx+(1-t)a\right) dt  \notag \\
&&  \notag \\
&=&\left. t^{\alpha }\frac{f\left( tx+(1-t)a\right) }{x-a}\right\vert
_{0}^{1}-\int_{0}^{1}\alpha t^{\alpha -1}\frac{f\left( tx+(1-t)a\right) }{x-a%
}dt  \notag \\
&&  \label{E2} \\
&=&\frac{f(x)}{x-a}-\frac{\alpha }{x-a}\int_{a}^{x}\frac{\left( a-u\right)
^{\alpha -1}}{\left( a-x\right) ^{\alpha -1}}\frac{f(u)}{x-a}du  \notag \\
&&  \notag \\
&=&\frac{f(x)}{x-a}-\frac{\alpha \Gamma (\alpha )}{\left( x-a\right)
^{\alpha +1}}\frac{1}{\Gamma (\alpha )}\int_{a}^{x}\left( u-a\right)
^{\alpha -1}f(u)du  \notag
\end{eqnarray}%
and 
\begin{eqnarray}
&&\int_{0}^{1}t^{\alpha }f^{\prime }\left( tx+(1-t)b\right) dt  \notag \\
&&  \notag \\
&=&\left. t^{\alpha }\frac{f\left( tx+(1-t)b\right) }{x-b}\right\vert
_{0}^{1}-\int_{0}^{1}\alpha t^{\alpha -1}\frac{f\left( tx+(1-t)b\right) }{x-b%
}dt  \notag \\
&&  \label{E3} \\
&=&\frac{f(x)}{x-b}-\frac{\alpha }{x-b}\int_{b}^{x}\frac{\left( b-u\right)
^{\alpha -1}}{\left( b-x\right) ^{\alpha -1}}\frac{f(u)}{x-b}du  \notag \\
&&  \notag \\
&=&\frac{f(x)}{x-b}+\frac{\alpha \Gamma (\alpha )}{\left( b-x\right)
^{\alpha +1}}\frac{1}{\Gamma (\alpha )}\int_{x}^{b}\left( b-u\right)
^{\alpha -1}f(u)du.  \notag
\end{eqnarray}%
Multiplying the both sides of (\ref{E2}) and (\ref{E3}) by $\frac{\left(
x-a\right) ^{\alpha +1}}{b-a}$ and $\frac{\left( b-x\right) ^{\alpha +1}}{b-a%
}$, respectively, we have 
\begin{equation}
\frac{\left( x-a\right) ^{\alpha +1}}{b-a}\int_{0}^{1}t^{\alpha }f^{\prime
}\left( tx+(1-t)a\right) dt=\frac{\left( x-a\right) ^{\alpha }f(x)}{b-a}-%
\frac{\Gamma (\alpha +1)}{b-a}J_{x-}^{\alpha }f(a)  \label{E4}
\end{equation}%
and 
\begin{equation}
\frac{\left( b-x\right) ^{\alpha +1}}{b-a}\int_{0}^{1}t^{\alpha }f^{\prime
}\left( tx+(1-t)b\right) dt=-\frac{\left( b-x\right) ^{\alpha }f(x)}{b-a}+%
\frac{\Gamma (\alpha +1)}{b-a}J_{x+}^{\alpha }f(b).  \label{E5}
\end{equation}%
From (\ref{E4}) and (\ref{E5}), it is obtained desired result.
\end{proof}

Using Lemma \ref{L1}, we can obtain the following fractional integral
inequalities:

\begin{theorem}
\label{T1} Let $f:\left[ a,b\right] \subset \lbrack 0,\infty )\rightarrow 
\mathbb{R}$, be a differentiable mapping on $(a,b)$ with $a<b$ such that $%
f^{\prime }\in L\left[ a,b\right] .$ If $\left\vert f^{\prime }\right\vert $
is $s-$convex in the second sense on $[a,b]$ for some fixed $s\in (0,1]$ and 
$\left\vert f^{\prime }(x)\right\vert \leq M,$ $x\in \left[ a,b\right] ,$
then the following inequality for fractional integrals with $\alpha >0$
holds:%
\begin{eqnarray}
&&\left\vert \left( \frac{\left( x-a\right) ^{\alpha }+\left( b-x\right)
^{\alpha }}{b-a}\right) f(x)-\frac{\Gamma (\alpha +1)}{\left( b-a\right) }%
\left[ J_{x-}^{\alpha }f(a)+J_{x+}^{\alpha }f(b)\right] \right\vert  \notag
\\
&&  \label{E6} \\
&\leq &\frac{M}{b-a}\left( 1+\frac{\Gamma (\alpha +1)\Gamma \left(
s+1\right) }{\Gamma (\alpha +s+1)}\right) \left[ \frac{\left( x-a\right)
^{\alpha +1}+\left( b-x\right) ^{\alpha +1}}{\alpha +s+1}\right]  \notag
\end{eqnarray}%
where $\Gamma $ is Euler Gamma function.
\end{theorem}

\begin{proof}
From (\ref{E1}) and since $\left\vert f^{\prime }\right\vert $ is a $s-$%
convex mapping in the second sense on $[a,b]$, we have%
\begin{eqnarray*}
&&\left\vert \left( \frac{\left( x-a\right) ^{\alpha }+\left( b-x\right)
^{\alpha }}{b-a}\right) f(x)-\frac{\Gamma (\alpha +1)}{\left( b-a\right) }%
\left[ J_{x-}^{\alpha }f(a)+J_{x+}^{\alpha }f(b)\right] \right\vert  \\
&& \\
&\leq &\frac{\left( x-a\right) ^{\alpha +1}}{b-a}\int_{0}^{1}t^{\alpha
}\left\vert f^{\prime }\left( tx+(1-t)a\right) \right\vert dt \\
&& \\
&&+\frac{\left( b-x\right) ^{\alpha +1}}{b-a}\int_{0}^{1}t^{\alpha
}\left\vert f^{\prime }\left( tx+(1-t)b\right) \right\vert dt \\
&& \\
&\leq &\frac{\left( x-a\right) ^{\alpha +1}}{b-a}\int_{0}^{1}t^{\alpha
+s}\left\vert f^{\prime }(x)\right\vert +t^{\alpha }(1-t)^{s}\left\vert
f^{\prime }(a)\right\vert dt \\
&& \\
&&+\frac{\left( b-x\right) ^{\alpha +1}}{b-a}\int_{0}^{1}t^{\alpha
+s}\left\vert f^{\prime }(x)\right\vert +t^{\alpha }(1-t)^{s}\left\vert
f^{\prime }(b)\right\vert dt \\
&& \\
&\leq &\frac{M}{b-a}\left( \frac{1}{\alpha +s+1}+\frac{\Gamma (\alpha
+1)\Gamma (s+1)}{\Gamma (\alpha +s+2)}\right) \left[ \left( x-a\right)
^{\alpha +1}+\left( b-x\right) ^{\alpha +1}\right] 
\end{eqnarray*}%
where we have used the fact that 
\begin{equation*}
\int_{0}^{1}t^{\alpha +s}dt=\frac{1}{\alpha +s+1}\text{ \ \ \ \ and \ \ \ \ }%
\int_{0}^{1}t^{\alpha }(1-t)^{s}dt=\frac{\Gamma (\alpha +1)\Gamma (s+1)}{%
\Gamma (\alpha +s+2)}.
\end{equation*}%
Hence, using the reduction formula $\Gamma (n+1)=n\Gamma (n)$ $(n>0)$ for
Euler Gamma function, the proof is complete.
\end{proof}

\begin{remark}
\label{R1} In Theorem \ref{T1}, if we choose $\alpha =1$ , then (\ref{E6})
reduces the inequality (\ref{e.1.4}) of Theorem \ref{t3}.
\end{remark}

\begin{theorem}
\label{T2} Let $f:\left[ a,b\right] \subset \lbrack 0,\infty )\rightarrow 
\mathbb{R}$, be a differentiable mapping on $(a,b)$ with $a<b$ such that $%
f^{\prime }\in L\left[ a,b\right] .$ If $\left\vert f^{\prime }\right\vert
^{q}$ is $s-$convex in the second sense on $[a,b]$ for some fixed $s\in (0,1]
$, $p,q>1$ and $\left\vert f^{\prime }(x)\right\vert \leq M,$ $x\in \left[
a,b\right] ,$ then the following inequality for fractional integrals holds:%
\begin{eqnarray}
&&\left\vert \left( \frac{\left( x-a\right) ^{\alpha }+\left( b-x\right)
^{\alpha }}{b-a}\right) f(x)-\frac{\Gamma (\alpha +1)}{\left( b-a\right) }%
\left[ J_{x-}^{\alpha }f(a)+J_{x+}^{\alpha }f(b)\right] \right\vert   \notag
\\
&&  \label{E7} \\
&\leq &\frac{M}{\left( 1+p\alpha \right) ^{\frac{1}{p}}}\left( \frac{2}{s+1}%
\right) ^{\frac{1}{q}}\left[ \frac{\left( x-a\right) ^{\alpha +1}+\left(
b-x\right) ^{\alpha +1}}{b-a}\right] ,  \notag
\end{eqnarray}%
where $\frac{1}{p}+\frac{1}{q}=1,$ $\alpha >0$ and $\Gamma $ is Euler Gamma
function.
\end{theorem}

\begin{proof}
From Lemma \ref{L1} and using the well known H\"{o}lder inequality , we have%
\begin{eqnarray*}
&&\left\vert \left( \frac{\left( x-a\right) ^{\alpha }+\left( b-x\right)
^{\alpha }}{b-a}\right) f(x)-\frac{\Gamma (\alpha +1)}{\left( b-a\right) }%
\left[ J_{x-}^{\alpha }f(a)+J_{x+}^{\alpha }f(b)\right] \right\vert  \\
&& \\
&\leq &\frac{\left( x-a\right) ^{\alpha +1}}{b-a}\int_{0}^{1}t^{\alpha
}\left\vert f^{\prime }\left( tx+(1-t)a\right) \right\vert dt \\
&& \\
&&+\frac{\left( b-x\right) ^{\alpha +1}}{b-a}\int_{0}^{1}t^{\alpha
}\left\vert f^{\prime }\left( tx+(1-t)b\right) \right\vert dt \\
&& \\
&\leq &\frac{\left( x-a\right) ^{\alpha +1}}{b-a}\left(
\int_{0}^{1}t^{p\alpha }dt\right) ^{\frac{1}{p}}\left(
\int_{0}^{1}\left\vert f^{\prime }\left( tx+(1-t)a\right) \right\vert
^{q}dt\right) ^{\frac{1}{q}} \\
&& \\
&&+\frac{\left( b-x\right) ^{\alpha +1}}{b-a}\left( \int_{0}^{1}t^{p\alpha
}dt\right) ^{\frac{1}{p}}\left( \int_{0}^{1}\left\vert f^{\prime }\left(
tx+(1-t)b\right) \right\vert ^{q}dt\right) ^{\frac{1}{q}}.
\end{eqnarray*}%
Since $\left\vert f^{\prime }\right\vert ^{q}$ is $s-$convex in the second
sense on $[a,b]$ and $\left\vert f^{\prime }(x)\right\vert \leq M$, we get
(see \cite[p. 1073]{ADDC})%
\begin{equation*}
\int_{0}^{1}\left\vert f^{\prime }\left( tx+(1-t)a\right) \right\vert
^{q}dt\leq \frac{2M^{q}}{s+1}\text{ \ and \ }\int_{0}^{1}\left\vert
f^{\prime }\left( tx+(1-t)b\right) \right\vert ^{q}dt\leq \frac{2M^{q}}{s+1}
\end{equation*}%
and by simple computation%
\begin{equation*}
\int_{0}^{1}t^{p\alpha }dt=\frac{1}{p\alpha +1}.
\end{equation*}%
Hence, we have 
\begin{eqnarray*}
&&\left\vert \left( \frac{\left( x-a\right) ^{\alpha }+\left( b-x\right)
^{\alpha }}{b-a}\right) f(x)-\frac{\Gamma (\alpha +1)}{\left( b-a\right) }%
\left[ J_{x-}^{\alpha }f(a)+J_{x+}^{\alpha }f(b)\right] \right\vert  \\
&& \\
&\leq &\frac{M}{\left( 1+p\alpha \right) ^{\frac{1}{p}}}\left( \frac{2}{s+1}%
\right) ^{\frac{1}{q}}\left[ \frac{\left( x-a\right) ^{\alpha +1}+\left(
b-x\right) ^{\alpha +1}}{b-a}\right] 
\end{eqnarray*}%
which completes the proof.
\end{proof}

\begin{remark}
\label{R2} In Theorem \ref{T2}, if we choose $\alpha =1$ , then (\ref{E7})
reduces the inequality (\ref{1.4.5}) of Theorem \ref{t4}.
\end{remark}

\begin{theorem}
\label{T3} Let $f:\left[ a,b\right] \subset \lbrack 0,\infty )\rightarrow 
\mathbb{R}$, be a differentiable mapping on $(a,b)$ with $a<b$ such that $%
f^{\prime }\in L\left[ a,b\right] .$ If $\left\vert f^{\prime }\right\vert
^{q}$ is $s-$convex in the second sense on $[a,b]$ for some fixed $s\in (0,1]
$, $q\geq 1,$ and $\left\vert f^{\prime }(x)\right\vert \leq M,$ $x\in \left[
a,b\right] ,$ then the following inequality for fractional integrals holds:%
\begin{eqnarray}
&&\left\vert \left( \frac{\left( x-a\right) ^{\alpha }+\left( b-x\right)
^{\alpha }}{b-a}\right) f(x)-\frac{\Gamma (\alpha +1)}{\left( b-a\right) }%
\left[ J_{x-}^{\alpha }f(a)+J_{x+}^{\alpha }f(b)\right] \right\vert   \notag
\\
&&  \label{E8} \\
&\leq &\frac{M}{\left( 1+p\alpha \right) ^{\frac{1}{p}}}\left( \frac{2}{s+1}%
\right) ^{\frac{1}{q}}\left[ \frac{\left( x-a\right) ^{\alpha +1}+\left(
b-x\right) ^{\alpha +1}}{b-a}\right] ,  \notag
\end{eqnarray}%
where $\alpha >0$ and $\Gamma $ is Euler Gamma function.
\end{theorem}

\begin{proof}
From Lemma \ref{L1} and using the well known power mean inequality, we have%
\begin{eqnarray*}
&&\left\vert \left( \frac{\left( x-a\right) ^{\alpha }+\left( b-x\right)
^{\alpha }}{b-a}\right) f(x)-\frac{\Gamma (\alpha +1)}{\left( b-a\right) }%
\left[ J_{x-}^{\alpha }f(a)+J_{x+}^{\alpha }f(b)\right] \right\vert  \\
&& \\
&\leq &\frac{\left( x-a\right) ^{\alpha +1}}{b-a}\int_{0}^{1}t^{\alpha
}\left\vert f^{\prime }\left( tx+(1-t)a\right) \right\vert dt \\
&& \\
&&+\frac{\left( b-x\right) ^{\alpha +1}}{b-a}\int_{0}^{1}t^{\alpha
}\left\vert f^{\prime }\left( tx+(1-t)b\right) \right\vert dt \\
&& \\
&\leq &\frac{\left( x-a\right) ^{\alpha +1}}{b-a}\left(
\int_{0}^{1}t^{\alpha }dt\right) ^{1-\frac{1}{q}}\left(
\int_{0}^{1}t^{\alpha }\left\vert f^{\prime }\left( tx+(1-t)a\right)
\right\vert ^{q}dt\right) ^{\frac{1}{q}} \\
&& \\
&&+\frac{\left( b-x\right) ^{\alpha +1}}{b-a}\left( \int_{0}^{1}t^{\alpha
}dt\right) ^{1-\frac{1}{q}}\left( \int_{0}^{1}t^{\alpha }\left\vert
f^{\prime }\left( tx+(1-t)b\right) \right\vert ^{q}dt\right) ^{\frac{1}{q}}.
\end{eqnarray*}%
Since $\left\vert f^{\prime }\right\vert ^{q}$ is $s-$convex in the second
sense on $[a,b]$ and $\left\vert f^{\prime }(x)\right\vert \leq M$, we get
(see \cite[p. 1073]{ADDC})%
\begin{eqnarray*}
\int_{0}^{1}t^{\alpha }\left\vert f^{\prime }\left( tx+(1-t)a\right)
\right\vert ^{q}dt &\leq &\int_{0}^{1}\left[ t^{s+\alpha }\left\vert
f^{\prime }\left( x\right) \right\vert ^{q}+t^{\alpha }(1-t)^{s}\left\vert
f^{\prime }\left( a\right) \right\vert ^{q}\right] dt \\
&& \\
&=&\frac{\left\vert f^{\prime }\left( x\right) \right\vert ^{q}}{\alpha +s+1}%
+\left\vert f^{\prime }\left( a\right) \right\vert ^{q}\int_{0}^{1}t^{\alpha
}(1-t)^{s}dt \\
&& \\
&=&\frac{\left\vert f^{\prime }\left( x\right) \right\vert ^{q}}{\alpha +s+1}%
+\left\vert f^{\prime }\left( a\right) \right\vert ^{q}\beta (\alpha +1,s+1)
\\
&& \\
&=&\frac{\left\vert f^{\prime }\left( x\right) \right\vert ^{q}}{\alpha +s+1}%
+\left\vert f^{\prime }\left( a\right) \right\vert ^{q}\frac{\Gamma (\alpha
+1)\Gamma (s+1)}{\left( \alpha +s+1\right) \Gamma (\alpha +s+1)} \\
&& \\
&\leq &\frac{M^{q}}{\alpha +s+1}\left( 1+\frac{\Gamma (\alpha +1)\Gamma (s+1)%
}{\Gamma (\alpha +s+1)}\right) 
\end{eqnarray*}%
and similarly 
\begin{eqnarray*}
\int_{0}^{1}t^{\alpha }\left\vert f^{\prime }\left( tx+(1-t)b\right)
\right\vert ^{q}dt &\leq &\int_{0}^{1}\left[ t^{s+\alpha }\left\vert
f^{\prime }\left( x\right) \right\vert ^{q}+t^{\alpha }(1-t)^{s}\left\vert
f^{\prime }\left( b\right) \right\vert ^{q}\right] dt \\
&& \\
&=&\frac{\left\vert f^{\prime }\left( x\right) \right\vert ^{q}}{\alpha +s+1}%
+\left\vert f^{\prime }\left( b\right) \right\vert ^{q}\int_{0}^{1}t^{\alpha
}(1-t)^{s}dt \\
&& \\
&=&\frac{\left\vert f^{\prime }\left( x\right) \right\vert ^{q}}{\alpha +s+1}%
+\left\vert f^{\prime }\left( b\right) \right\vert ^{q}\beta (\alpha +1,s+1)
\\
&& \\
&=&\frac{\left\vert f^{\prime }\left( x\right) \right\vert ^{q}}{\alpha +s+1}%
+\left\vert f^{\prime }\left( b\right) \right\vert ^{q}\frac{\Gamma (\alpha
+1)\Gamma (s+1)}{\left( \alpha +s+1\right) \Gamma (\alpha +s+1)} \\
&& \\
&\leq &\frac{M^{q}}{\alpha +s+1}\left( 1+\frac{\Gamma (\alpha +1)\Gamma (s+1)%
}{\Gamma (\alpha +s+1)}\right) ,
\end{eqnarray*}%
where $\beta $ is Euler Beta function defined by 
\begin{equation*}
\beta (x,y)=\int_{0}^{1}t^{x-1}(1-t)^{y-1}dt\text{ \ \ }(x,y>0).
\end{equation*}%
We used the fact that 
\begin{equation*}
\beta (x,y)=\frac{\Gamma (x)\Gamma (y)}{\Gamma (x+y)}\text{ \ \ and \ \ }%
\Gamma (n+1)=n\Gamma (n)(n>0).
\end{equation*}%
Hence, we have 
\begin{eqnarray*}
&&\left\vert \left( \frac{\left( x-a\right) ^{\alpha }+\left( b-x\right)
^{\alpha }}{b-a}\right) f(x)-\frac{\Gamma (\alpha +1)}{\left( b-a\right) }%
\left[ J_{x-}^{\alpha }f(a)+J_{x+}^{\alpha }f(b)\right] \right\vert  \\
&& \\
&\leq &M\left( \frac{1}{1+\alpha }\right) ^{1-\frac{1}{q}}\left( \frac{1}{%
\alpha +s+1}\right) ^{\frac{1}{q}} \\
&&\times \left( 1+\frac{\Gamma (\alpha +1)\Gamma (s+1)}{\Gamma (\alpha +s+1)}%
\right) ^{\frac{1}{q}}\left[ \frac{\left( x-a\right) ^{\alpha +1}+\left(
b-x\right) ^{\alpha +1}}{b-a}\right] 
\end{eqnarray*}%
which completes the proof.
\end{proof}

\begin{remark}
\label{R3} In Theorem \ref{T3}, if we choose $\alpha =1$ , then (\ref{E8})
reduces the inequality (\ref{1.4.6}) of Theorem \ref{t5}.
\end{remark}

The following result holds for s-concavity:

\begin{theorem}
\label{T4} Let $f:\left[ a,b\right] \subset \lbrack 0,\infty )\rightarrow 
\mathbb{R}$, be a differentiable mapping on $(a,b)$ with $a<b$ such that $%
f^{\prime }\in L\left[ a,b\right] .$ If $\left\vert f^{\prime }\right\vert
^{q}$ is $s-$concave in the second sense on $[a,b]$ for some fixed $s\in
(0,1]$ and $p,q>1,$ then the following inequality for fractional integrals
holds:%
\begin{eqnarray}
&&\left\vert \left( \frac{\left( x-a\right) ^{\alpha }+\left( b-x\right)
^{\alpha }}{b-a}\right) f(x)-\frac{\Gamma (\alpha +1)}{\left( b-a\right) }%
\left[ J_{x-}^{\alpha }f(a)+J_{x+}^{\alpha }f(b)\right] \right\vert   \notag
\\
&&  \label{E9} \\
&\leq &\frac{2^{\left( s-1\right) /q}}{\left( 1+p\alpha \right) ^{\frac{1}{p}%
}\left( b-a\right) }\left[ \left( x-a\right) ^{\alpha +1}\left\vert
f^{\prime }\left( \frac{x+a}{2}\right) \right\vert +\left( b-x\right)
^{\alpha +1}\left\vert f^{\prime }\left( \frac{b+x}{2}\right) \right\vert %
\right] ,  \notag
\end{eqnarray}%
where $\frac{1}{p}+\frac{1}{q}=1,$ $\alpha >0$ and $\Gamma $ is Euler Gamma
function.
\end{theorem}

\begin{proof}
From Lemma \ref{L1} and using the well known H\"{o}lder inequality , we have%
\begin{eqnarray}
&&\left\vert \left( \frac{\left( x-a\right) ^{\alpha }+\left( b-x\right)
^{\alpha }}{b-a}\right) f(x)-\frac{\Gamma (\alpha +1)}{\left( b-a\right) }%
\left[ J_{x-}^{\alpha }f(a)+J_{x+}^{\alpha }f(b)\right] \right\vert   \notag
\\
&&  \notag \\
&\leq &\frac{\left( x-a\right) ^{\alpha +1}}{b-a}\int_{0}^{1}t^{\alpha
}\left\vert f^{\prime }\left( tx+(1-t)a\right) \right\vert dt  \notag \\
&&  \notag \\
&&+\frac{\left( b-x\right) ^{\alpha +1}}{b-a}\int_{0}^{1}t^{\alpha
}\left\vert f^{\prime }\left( tx+(1-t)b\right) \right\vert dt  \label{E10} \\
&&  \notag \\
&\leq &\frac{\left( x-a\right) ^{\alpha +1}}{b-a}\left(
\int_{0}^{1}t^{p\alpha }dt\right) ^{\frac{1}{p}}\left(
\int_{0}^{1}\left\vert f^{\prime }\left( tx+(1-t)a\right) \right\vert
^{q}dt\right) ^{\frac{1}{q}}  \notag \\
&&  \notag \\
&&+\frac{\left( b-x\right) ^{\alpha +1}}{b-a}\left( \int_{0}^{1}t^{p\alpha
}dt\right) ^{\frac{1}{p}}\left( \int_{0}^{1}\left\vert f^{\prime }\left(
tx+(1-t)b\right) \right\vert ^{q}dt\right) ^{\frac{1}{q}}.  \notag
\end{eqnarray}%
Since $\left\vert f^{\prime }\right\vert ^{q}$ is $s-$concave, using the
inequality (\ref{e.1.3}) we get (see \cite[p. 1074]{ADDC})%
\begin{equation}
\int_{0}^{1}\left\vert f^{\prime }\left( tx+(1-t)a\right) \right\vert
^{q}dt\leq 2^{s-1}\left\vert f^{\prime }\left( \frac{x+a}{2}\right)
\right\vert ^{q}  \label{E11}
\end{equation}%
and%
\begin{equation}
\int_{0}^{1}\left\vert f^{\prime }\left( tx+(1-t)b\right) \right\vert
^{q}dt\leq 2^{s-1}\left\vert f^{\prime }\left( \frac{b+x}{2}\right)
\right\vert ^{q}.  \label{E12}
\end{equation}%
Using (\ref{E11}) and (\ref{E12}) in (\ref{E10}), we have 
\begin{eqnarray*}
&&\left\vert \left( \frac{\left( x-a\right) ^{\alpha }+\left( b-x\right)
^{\alpha }}{b-a}\right) f(x)-\frac{\Gamma (\alpha +1)}{\left( b-a\right) }%
\left[ J_{x-}^{\alpha }f(a)+J_{x+}^{\alpha }f(b)\right] \right\vert  \\
&& \\
&\leq &\frac{2^{\left( s-1\right) /q}}{\left( 1+p\alpha \right) ^{\frac{1}{p}%
}\left( b-a\right) }\left[ \left( x-a\right) ^{\alpha +1}\left\vert
f^{\prime }\left( \frac{x+a}{2}\right) \right\vert +\left( b-x\right)
^{\alpha +1}\left\vert f^{\prime }\left( \frac{b+x}{2}\right) \right\vert %
\right] 
\end{eqnarray*}%
which completes the proof.
\end{proof}

\begin{remark}
\label{R4} In Theorem \ref{T4}, if we choose $\alpha =1$ , then (\ref{E9})
reduces the inequality (\ref{1.4.7}) of Theorem \ref{t6}.
\end{remark}


\begin{thebibliography}{99}
\bibitem{D1} S.S. Dragomir, On the Ostrowski's integral inequality for
mappings with bounded variation and applications, Math. Ineq. \&Appl., 1(2)
(1998).

\bibitem{D3} S.S. Dragomir, The Ostrowski integral inequality for
Lipschitzian mappings and applications, Comput. Math. Appl., 38 (1999),
33-37.

\bibitem{DW} S. S. Dragomir, S. Wang, A new inequality of Ostrowski's type
in $L_{1}$-norm and applications to some special means and to some numerical
quadrature rules, Tamkang J. of Math., 28 (1997), 239--244.

\bibitem{Liu} Z. Liu, Some companions of an Ostrowski type inequality and
application, J. Inequal. in Pure and Appl. Math, 10(2), 2009, Art. 52, 12 pp.

\bibitem{Ozdemir} M.E. \"{O}zdemir, H. Kavurmac\i\ and E. Set, Ostrowski's
type inequalities for $(\alpha ,m)-$convex functions, Kyungpook Math. J., 50
(2010), 371-378.

\bibitem{Pachpatte} B. G. Pachpatte, On an inequality of Ostrowski type in
three independent variables, J. Math.Anal. Appl., 249(2000), 583-591.

\bibitem{Pachpatte1} B. G. Pachpatte, On a new Ostrowski type inequality in
two independent variables, Tamkang J. Math., 32(1), (2001), 45-49

\bibitem{Rafýq} A. Rafiq, N.A. Mir and F. Ahmad, Weighted \v{C}eby\v{s}%
ev-Ostrowski type inequalities, Applied Math. Mechanics (English Edition),
2007, 28(7), 901-906.

\bibitem{sarikaya1} M. Z. Sarikaya, On the Ostrowski type integral
inequality, Acta Math. Univ. Comenianae, Vol. LXXIX, 1(2010), pp. 129-134.

\bibitem{Ujevic1} N. Ujevi\'{c}, Sharp inequalities of Simpson type and
Ostrowski type, Comput. Math. Appl. 48 (2004) 145-151.

\bibitem{Zhongxue} L. Zhongxue, On sharp inequalities of Simpson type and
Ostrowski type in two independent variables, Comput. Math. Appl., 56 (2008)
2043-2047.

\bibitem{Alomari1} M. Alomari and M. Darus, Some Ostrowski type inequalities
for convex functions with applications, RGMIA 13 (1) (2010) article No. 3.
Preprint.

\bibitem{ADDC} M. Alomari, M. Darus, S.S. Dragomir, P. Cerone, Ostrowski
type inequalities for functions whose derivatives are s-convex in the second
sense, Appl. Math. Lett. 23 (2010) 1071-1076.

\bibitem{dragomir3} S. S. Dragomir and S. Fitzpatrik, The Hadamard's
inequality for\ $s$-convex functions in the second sense, Demonstratio Math.
32(4), (1999), 687-696.

\bibitem{hudzik} H. Hudzik and L. Maligranda, Some remarks on $s-$convex
functions, Aequationes Math. 48 (1994), 100-111.

\bibitem{KBOP} U.S. Kirmaci, M.K. Bakula, M.E. \"{O}zdemir, J. Pe\v{c}ari%
\'{c}, Hadamard-tpye inequalities for $s$-convex functions, Appl. Math. and
Comp., 193 (2007), 26-35.

\bibitem{SSO} M.Z. Sarikaya, E. Set and M.E. \"{O}zdemir "On new
inequalities of Simpson's type for s-convex functions", Comput. Math. Appl.,
60(8), (2010) 2191-2199.

\bibitem{Anastassiou} G. Anastassiou, M.R. Hooshmandasl, A. Ghasemi and F.
Moftakharzadeh, Montogomery identities for fractional integrals and related
fractional inequalities, J. Ineq. Pure and Appl. Math., 10(4) (2009), Art.
97.

\bibitem{Belarbi} S. Belarbi and Z. Dahmani, On some new fractional integral
inequalities, J. Ineq. Pure and Appl. Math., 10(3) (2009), Art. 86.

\bibitem{Dahmani1} Z. Dahmani, New inequalities in fractional integrals,
International Journal of Nonlinear Scinece, 9(4) (2010), 493-497.

\bibitem{Dahmani2} Z. Dahmani, On Minkowski and Hermite-Hadamard integral
inequalities via fractional integration, Ann. Funct. Anal. 1(1) (2010),
51-58.

\bibitem{Dahmani3} Z. Dahmani, L. Tabharit, S. Taf, Some fractional integral
inequalities, Nonl. Sci. Lett. A, 1(2) (2010), 155-160.

\bibitem{Dahmani4} Z. Dahmani, L. Tabharit, S. Taf, New generalizations of
Gruss inequality usin Riemann-Liouville fractional integrals, Bull. Math.
Anal. Appl., 2(3) (2010), 93-99.

\bibitem{Gorenflo} R. Gorenflo, F. Mainardi, Fractional calculus: integral
and differential equations of fractional order, Springer Verlag, Wien
(1997), 223-276.

\bibitem{Miller} S. Miller and B. Ross, An introduction to the Fractional
Calculus and Fractional Differential Equations, John Wiley \& Sons, USA,
1993, p.2.

\bibitem{Podlubni} I. Podlubni, Fractional Differential Equations, Academic
Press, San Diego, 1999.

\bibitem{sarikaya} M.Z. Sar\i kaya, E. Set, H. Yald\i z and N. Ba\c{s}ak,
Hermite-Hadamard's inequalities for fractional integrals and related
fractional inequalities, Submitted.

\bibitem{sarikaya2} M.Z. Sarikaya and H. Ogunmez, On new inequalities via
Riemann-Liouville fractional integration, arXiv:1005.1167v1, submitted.

\bibitem{Mitrinovic1} D.S. Mitrinovi\'{c}, J.E. Pe\v{c}ari\'{c} and A.M.
Fink, Inequalities Involving Functions and Their Integrals and Derivatives,
Kluwer Academic Publishers, Dortrecht, 1991

\bibitem{OST} A.M. Ostrowski, \"{U}ber die Absolutabweichung einer
differentierbaren Funktion von ihren Integralmittelwert, Comment. Math.
Helv., 10, 226-227, (1938).

\bibitem{mitrinovic} D.S. Mitrinovi\'{c}, J.E. Pe\v{c}ari\'{c} and A.M.
Fink, Classical and New Inequalities in Analysis, Kluwer Academic
Publishers, Dordrecht, 1993, p. 106.
\end{thebibliography}
\end{document}